\documentclass[a4paper,11pt]{amsart}

\usepackage{a4wide}

\usepackage[foot]{amsaddr}

\newtheorem{theorem}{Theorem}
\newtheorem{lemma}[theorem]{Lemma}
\newtheorem{corollary}[theorem]{Corollary}
\newtheorem{remark}[theorem]{Remark}

\usepackage[breaklinks,bookmarks=false]{hyperref}
\hypersetup{colorlinks, linkcolor=blue, citecolor=blue,urlcolor=blue,plainpages=false}

\usepackage{amssymb}
\usepackage{mathrsfs}

\newcommand{\eq}{:=}

\newcommand{\grad}{\boldsymbol \nabla}
\renewcommand{\div}{\grad \cdot}
\newcommand{\curl}{\grad \times}

\newcommand{\ccurl}{\boldsymbol{\operatorname{curl}}}
\newcommand{\ddiv}{\operatorname{div}}

\newcommand{\BD}{\boldsymbol D}

\newcommand{\BH}{\boldsymbol H}
\newcommand{\BI}{\boldsymbol I}
\newcommand{\BJ}{\boldsymbol J}

\newcommand{\BL}{\boldsymbol L}

\newcommand{\BN}{\boldsymbol N}

\newcommand{\BP}{\boldsymbol P}

\newcommand{\BW}{\boldsymbol W}

\newcommand{\be}{\boldsymbol e}
\newcommand{\bg}{\boldsymbol g}

\newcommand{\bn}{\boldsymbol n}
\newcommand{\bo}{\boldsymbol o}

\newcommand{\br}{\boldsymbol r}

\newcommand{\bu}{\boldsymbol u}
\newcommand{\bv}{\boldsymbol v}
\newcommand{\bw}{\boldsymbol w}
\newcommand{\bx}{\boldsymbol x}
\newcommand{\by}{\boldsymbol y}

\newcommand{\CJ}{\mathcal J}

\newcommand{\CT}{\mathcal T}

\newcommand{\LC}{\mathscr C}

\newcommand{\LF}{\mathscr F}

\newcommand{\LP}{\mathscr P}

\newcommand{\eps}{\varepsilon}

\newcommand{\ee}{{\boldsymbol \eps}}
\newcommand{\mm}{{\boldsymbol \mu}}
\newcommand{\cc}{{\boldsymbol \chi}}
\newcommand{\zz}{{\boldsymbol \zeta}}

\newcommand{\norm}[1]{|\!|\!|#1|\!|\!|}

\newcommand{\enorm}[1]{\norm{#1}_{\ccurl,\omega,\Omega}}

\newcommand{\bphi}{\boldsymbol \phi}

\newcommand{\cI}{\LC_{{\rm i},p}}
\newcommand{\cD}{\LC_{{\rm shift},p}}

\newcommand{\cSE}{c_{\rm s}}
\newcommand{\gbaE}{\gamma}

\newcommand{\pp}{\boldsymbol \varphi}
\newcommand{\al}{\boldsymbol \alpha}

\title%
[Approximability estimates for Maxwell's equations]%
{Frequency-explicit approximability estimates for time-harmonic Maxwell's equations}

\author{T. Chaumont-Frelet$^{\star,\dagger}$}
\author{P. Vega$^{\star,\dagger}$}

\address{\vspace{-.5cm}}
\address{\noindent \tiny \textup{$^\star$Inria, 2004 Route des Lucioles, 06902 Valbonne, France}}
\address{\noindent \tiny \textup{$^\dagger$Laboratoire J.A. Dieudonn\'e, Parc Valrose, 28 Avenue Valrose, 06108 Nice Cedex 02, 06000 Nice, France}}

\title[Approximability estimates for Maxwell's equations]{Frequency-explicit approximability estimates for time-harmonic Maxwell's equations}

\begin{document}

\begin{abstract}
We consider time-harmonic Maxwell's equations set in a heterogeneous medium
with perfectly conducting boundary conditions. Given a divergence-free right-hand
side lying in $L^2$, we provide a frequency-explicit approximability estimate measuring
the difference between the corresponding solution and its best approximation by high-order
N\'ed\'elec
finite elements. Such an approximability estimate is crucial in both the a priori and
a posteriori error analysis of finite element discretizations of Maxwell's equations,
but the derivation is not trivial. Indeed, it is hard to take advantage of high-order
polynomials given that the right-hand side only exhibits $L^2$ regularity. We proceed in
line with previously obtained results for the simpler setting of the scalar Helmholtz equation
and propose a regularity splitting of the solution. In turn, this splitting yields sharp
approximability estimates generalizing known results for the scalar Helmholtz equation and
showing the interest of high-order methods.

\vspace{.5cm}
\end{abstract}

\maketitle

\section{Introduction}

Over the past decades, considerable efforts have been devoted to analyze
the stability and convergence of finite element discretizations of high-frequency wave
propagation problems. This is in part because the required mathematical analysis
is rich and elegant, but also due to the large number of physical and industrial
applications for which these problems are relevant.

Except in the low-frequency regime, the bilinear (or sesquilinear) forms associated
with time-harmonic wave propagation problems are not coercive. As a consequence, finite
element schemes become unstable when the frequency is high and/or close to a resonant
frequency, unless heavily refined meshes or high-order elements are employed
\cite{ihlenburg_babuska_1997a}. For scalar wave propagation problems modeled by the Helmholtz
equation, it has become clear that high-order elements are very well-suited
to address these stability issues. On the one hand, the interest of high-order elements
has been numerically noted in a number of works
\cite{beriot_prinn_gabard_2016a,taus_zepedanunez_hewett_demanet_2017a}.
On the other hand, thanks to dedicated duality techniques, the stability and convergence
theory is now well-understood
\cite{chaumontfrelet_nicaise_2019a,%
ihlenburg_babuska_1997a,%
melenk_sauter_2011a},
and is in line with numerical observations.
Vectorial problems are less covered in the literature, but the few available results
point towards the fact that the analysis techniques employed for the Helmholtz
equation as well as the key conclusions can be extended
\cite{chaumontfrelet_2019a,melenk_sauter_2020a}.
We also mention \cite{%
chaumontfrelet_nicaise_pardo_2018a,%
ern_guermond_2018a,%
zhong_shu_wittum_xu_2009a},
where similar duality techniques are used for vectorial wave propagation problems,
without focusing on the high-frequency regime though.

Here, we consider time-harmonic Maxwell's equations
\begin{equation}
\label{eq_maxwell_strong}
\left \{
\begin{array}{rcll}
-\omega^2 \ee \be + \curl \left (\mm^{-1} \curl \be\right )
&=&
\omega \ee \bg
&
\text{ in } \Omega
\\
\be \times \bn &=& \boldsymbol 0 & \text{ on } \partial \Omega
\end{array}
\right .
\end{equation}
in a smooth domain $\Omega$ with piecewise smooth permittivity and permeability
(real-valued, symmetric) tensors $\ee$ and $\mm$. In \eqref{eq_maxwell_strong},
$\omega > 0$ is the frequency, $\be: \Omega \to \mathbb R^3$ is the unknown
and $\bg: \Omega \to \mathbb R^3$ is a given right-hand side.
In practice, the right-hand side takes the form $\bg = i\ee^{-1} \BJ$
where $\BJ: \Omega \to \mathbb C^3$ is a ``current density'', and $\be$
represents the electric fields \cite{assous_ciarlet_labrunie_2018a}. Here,
we choose to work with $\bg$ instead of $\BJ$ since it is more relevant
mathematically as it naturally appears in convergence analysis by duality.

For a finite element space $\BW_h$, we define the ``approximation factor'' as the
sharpest constant $\gamma$ such that the estimate
\begin{equation}
\label{eq_approximation_gamma}
\inf_{\bv_h \in \BW_h} \enorm{\be-\bv_h} \leq \gamma \|\bg\|_{\ee,\Omega}
\end{equation}
holds for all $\bg \in \BL^2(\Omega)$ with $\div(\ee\bg) = 0$, where
$\|\cdot\|_{\ee,\Omega}$ is the $\ee$-weighted $\BL^2(\Omega)$ norm and
$\enorm{\cdot}$ is a suitable ``energy'' norm
(see Section \ref{section_functional_spaces}). The approximation factor
$\gamma$ quantifies the ability of the finite element space $\BW_h$ to reproduce
solutions to \eqref{eq_maxwell_strong}. Actually, this quantity is central
in the stability analysis of the finite element schemes, since it can be shown that
the finite element solution is quasi-optimal if and only if
$\gamma$ is ``sufficiently small''
\cite{chaumontfrelet_2019a,chaumontfrelet_nicaise_2019a,melenk_sauter_2011a}.
The approximation factor also plays a central role in a posteriori
error estimation \cite{chaumontfrelet_ern_vohralik_2021a,sauter_zech_2015a}.

Since the norm in the right-hand side of \eqref{eq_approximation_gamma}
is weak, one cannot expect a high regularity for the solution $\be$. As
a result, taking advantage of high-order polynomials is a subtle task:
the solution is only piecewise $H^2$ in general, so that the optimal approximation
rate is of order $h$, and not $h^p$. One key idea to overcome this issue is to introduce
a regularity splitting as initially done in \cite{melenk_sauter_2011a} for scalar wave
propagation in homogeneous media and later extended to heterogeneous media
\cite{chaumontfrelet_nicaise_2019a,lafontaine_spence_wunsch_2020a}
and Maxwell's equations in homogeneous media \cite{melenk_sauter_2020a}.

In this work, we apply the idea of \cite{chaumontfrelet_nicaise_2019a}
to obtain a regularity splitting for Maxwell's equations in heterogeneous
media. Our key result in Corollary \ref{corollary_approximability}
is that if $\BW_h$ is the N\'ed\'elec finite element space of order $p \geq 0$
on a shape-regular mesh $\CT_h$ with maximal element size $h$, there exist
positive constants $c$ and $C$ independent of $\omega$ and $h$
such that if $\omega h/\vartheta_\Omega \leq c$, then
\begin{equation}
\label{eq_upper_bound_gamma}
\gamma \leq C
\left (
\frac{\omega h}{\vartheta_\Omega}
+
\frac{\omega}{\delta}
\left (
\frac{\omega h}{\vartheta_\Omega}
\right )^{p+1}
\right )
\end{equation}
where $\delta$ is the distance between $\omega$ and the closest resonant
frequency (see Section \ref{section_eigen}), and $\vartheta_\Omega$ is the smallest wavespeed in $\Omega$.

Since $N_\lambda \eq (\omega h/\vartheta_\Omega)^{-1}$ is a measure of
the number of mesh elements per wavelength, one sees from \eqref{eq_upper_bound_gamma}
that $\gamma$ stays small as long as $N_\lambda \geq C (1 + (\omega/\delta)^{1/(p+1)})$
%%%   \begin{equation*}
%%%   N_\lambda \geq C \left (1 + \left (\frac{\omega}{\delta}\right )^{1/(p+1)}\right )
%%%   \end{equation*}
with a constant $C$ independent of $\omega$ and $h$. As a result, while the 
number of elements per wavelength must be increased to achieved stability when
the frequency is high ($\omega d_\Omega/\vartheta_\Omega \gg 1$, $d_\Omega$ being
the diameter of $\Omega$) or almost resonant ($\delta \ll 1$), the requirement is
less demanding for high-order elements.

We close this introduction with two comments. (i)
The authors largely expect the upper bound in \eqref{eq_upper_bound_gamma} is sharp,
as the same estimate is valid and sharp in the simpler setting of the Helmholtz
equation \cite{chaumontfrelet_nicaise_2019a}. (ii) The constants $c$ and $C$ in
\eqref{eq_upper_bound_gamma} are allowed to depend on $p$, which is an important
limitation. Unfortunately, the authors do not believe that $p$-explicit results
can be obtained in the present setting. Indeed, it appears that $p$-explicit
approximability requires substantially more involved arguments that are not
available in heterogeneous media so far \cite{melenk_sauter_2011a,melenk_sauter_2020a}.

The remaining of this work is organized as follows. Section \ref{section_settings}
presents the notation and recalls key preliminary results. In Section
\ref{section_stability}, we present some initial results concerning the stability
of the problem and basic regularity results. We elaborate a regularity
splitting in Section \ref{section_splitting} that we subsequently apply
to derive our approximability result in Section \ref{section_approximability},
leading to estimate \eqref{eq_upper_bound_gamma}.

\section{Settings}
\label{section_settings}

\subsection{Domain and coefficients}

We consider a simply connected domain $\Omega \subset \mathbb R^3$ with an analytic boundary
$\partial \Omega$. $\Omega$ is partitioned into a set $\LP$ of non-overlapping subdomains
$P$ with analytic boundaries $\partial P$ such that
$\overline{\Omega} = \sup_{P \in \LP} \overline{P}$. The notation
$d_\Omega \eq \max_{\bx,\by \in \overline{\Omega}} |\bx-\by|$ stands
for the diameter of $\Omega$.

$\ee$ and $\mm$ are two real symmetric tensor-valued functions defined over $\Omega$.
These coefficients are assumed to be piecewise smooth in the sense that for
each $P \in \LP$ and for $1 \leq j,\ell \leq 3$, $\ee_{j\ell}|_P$ and
$\mm_{j\ell}|_P$ are analytic functions. The notations
$\zz \eq \ee^{-1}$ and $\cc \eq \mm^{-1}$ will be useful in the sequel.

We denote by $\eps_{\min},\eps_{\max}: \Omega \to \mathbb R$ the (analytic) functions
mapping to each $\bx \in \Omega$ the smallest and largest eigenvalue of $\ee(\bx)$,
and we assume that $\ee$ is uniformly bounded and elliptic, which means that
\begin{equation*}
0 < \inf_{\Omega} \eps_{\min}, \qquad \sup_{\Omega} \eps_{\max} < +\infty.
\end{equation*}
We employ similar notations for $\mm$, $\cc$ and $\zz$, and assume that $\mm$
is uniformly bounded and elliptic. Finally, we denote by
\begin{equation*}
\vartheta_\Omega
\eq
\inf_\Omega \frac{1}{\sqrt{\eps_{\max}\mu_{\max}}}
\end{equation*}
the smallest wavespeed in $\Omega$.

\subsection{Functional spaces}
\label{section_functional_spaces}

If $D \subset \Omega$ is an open set, $L^2(D)$ is the usual Lebesgue-space of
real-valued square-integrable functions over $D$.
In addition, we write $\BL^2(D) \eq \left (L^2(D)\right )^3$ for vector-valued functions.
The natural inner products and norms of both these spaces are $(\cdot,\cdot)_D$ and $\|\cdot\|_D$,
and we drop the subscript when $D = \Omega$. If $\pp$ is a measurable uniformly bounded and
elliptic symmetric tensor-valued function, we also employ the (equivalent) norm
\begin{equation*}
\|\bv\|^2_{\pp,D} \eq \int_D \pp \bv \cdot \bv
\end{equation*}
for $\bv \in \BL^2(D)$.

$\BH(\ccurl,\Omega)$ is the Sobolev space of vector-valued functions
$\bv \in \BL^2(\Omega)$ such that $\curl \bv \in \BL^2(\Omega)$. It is
equipped with the ``energy'' norm
\begin{equation*}
\enorm{\bv}^2 \eq \omega^2 \|\bv\|_{\ee,\Omega}^2 + \|\curl \bv\|_{\cc,\Omega}^2
\qquad
\forall \bv \in \BH(\ccurl,\Omega).
\end{equation*}
$\BH_0(\ccurl,\Omega)$ is the closure of smooth compactly supported functions into
$\BH(\ccurl,\Omega)$ and contains vector-valued functions with vanishing tangential traces.

If $\pp$ is a measurable, uniformly elliptic and bounded tensor-valued function,
$\BH(\ddiv^0,\pp,\Omega)$ is the space of functions $\bv \in \BL^2(\Omega)$ such that
$\div(\pp\bv) = 0$ in $\Omega$. We simply write $\BH(\ddiv^0,\Omega)$ when
$\pp \eq \BI$ is the identity tensor. Besides, the space $\BH_0(\ddiv^0,\pp,\Omega)$
is the closure of smooth compactly supported functions in $\BH(\ddiv^0,\pp,\Omega)$
and contains functions with vanishing normal traces.

If $m \geq 0$, the space $\BH^m(\LP)$ contains those functions $\bv \in \BL^2(\Omega)$ such
that for each $P \in \LP$, $1 \leq \ell \leq 3$ and all multi-indices
$\al \in \mathbb N^3$ with $|\al| \leq m$, we have $\partial^\alpha (\bv_\ell|_P) \in L^2(P)$.
We equip this space with the norms
\begin{equation*}
\|\bv\|_{\ee,\BH^m(\LP)}^2
\eq
\|\bv\|_{\ee,\Omega}^2
+
\sum_{n=1}^m \sum_{|\al| = n} \sum_{P \in \LP} \sum_{\ell=1}^3
d_\Omega^{2n} \int_P \eps_{\max} |\partial^{\al} (\bv_\ell|_P)|^2
\end{equation*}
and
\begin{equation*}
\|\bv\|_{\cc,\BH^m(\LP)}^2
\eq
\|\bv\|_{\cc,\Omega}^2
+
\sum_{n=1}^m \sum_{|\al| = n} \sum_{P \in \LP} \sum_{\ell=1}^3
d_\Omega^{2n} \int_P \chi_{\max} |\partial^{\al} (\bv_\ell|_P)|^2.
\end{equation*}

We refer the reader to \cite{adams_fournier_2003a} for a detailed exposition
concerning Lesbegue and high-order Sobolev spaces, and to \cite{fernandes_gilardi_1997a}
and \cite{girault_raviart_1986a} for Sobolev spaces involving the curl and
divergence of vector fields.

\subsection{Eigenpairs}
\label{section_eigen}

Recalling that $\Omega$ is simply connected, it follows
from \cite[Remark 7.5]{fernandes_gilardi_1997a} that the application
\begin{equation*}
\BH_0(\ccurl,\Omega) \ni \bv \to \|\curl \bv\|_{\cc,\Omega}
\end{equation*}
is a norm on $\BH_0(\ccurl,\Omega) \cap \BH(\ddiv^0,\ee,\Omega)$.
Besides, the injection
$\BH_0(\ccurl,\Omega) \cap \BH(\ddiv^0,\ee,\Omega) \subset \BH(\ddiv^0,\ee,\Omega)$
is compact \cite{weber_1980a}. As a result
(see, e.g. \cite[Theorem 4.5.11]{assous_ciarlet_labrunie_2018a}),
there exists an orthonormal basis
$\{\bphi_j\}_{j \geq 0}$ of $\BH(\ddiv^0,\ee,\Omega)$
(equipped with the inner-product $(\ee\ \!\cdot,\cdot)$) and a sequence of strictly
positive eigenvalues $\{\lambda_j\}_{j \geq 0}$ such that for all $j \geq 0$,
\begin{equation*}
(\cc\curl \bphi_j,\curl \bv) = \lambda_j (\ee\bphi_j\bv) \qquad \forall \bv \in \BH_0(\ccurl,\Omega).
\end{equation*}
In addition, if $\bv \in \BH_0(\ccurl,\Omega) \cap \BH(\ddiv^0,\ee,\Omega)$, we have
\begin{equation*}
\|\bv\|_{\ee,\Omega}^2 = \sum_{j \geq 0} |v_j|^2
\quad
\text{ and }
\quad
\|\curl \bv\|_{\cc,\Omega}^2 = \sum_{j \geq 0} \lambda_j |v_j|^2,
\end{equation*}
where $v_j \eq (\ee\bv,\bphi_j)$.

In the remaining of this work, we set $\delta \eq \min_{j \geq 0}|\sqrt{\lambda_j} - \omega|$
and assume $\delta > 0$.

\begin{remark}[What can be said about $\delta$?]
In practice, it is complicated to obtain a bound for $\delta$ analytically
because it requires information about the localization of the spectrum.
In the high-frequency regime, $\delta$ will, in general, tend to be smaller
due to Weyl's law \cite[Theorem 6.8]{zworski_2012a}. Alternatively, in the low-frequency
regime where $0 < \omega < \sqrt{\lambda_0}$, a lower bound for $\delta$ may be computed from
a lower bound on $\lambda_0$, see \cite{gallistl_olkhovskiy_2021}.
\end{remark}

\subsection{Regularity shifts}
\label{section_shift}

Our analysis heavily relies on regularity shift results where, given
a divergence-free vector field with a smooth curl, one deduces smoothness
results for the field itself. First \cite[Theorem 2.2]{weber_1981a},
for all $p \geq 0$, there exists a constant $\cD$ only depending on $p$, $\LP$, $\ee$ and $\mm$
such that, for $0 \leq m \leq p$, if $\bv \in \BH_0(\ccurl,\Omega) \cap \BH(\ddiv^0,\ee,\Omega)$
with $\curl \bv \in \BH^m(\LP)$, we have
\begin{equation}
\label{eq_shift_ee}
\|\bv\|_{\ee,\BH^{m+1}(\LP)}
\leq
\cD \frac{d_\Omega}{\vartheta_\Omega} \|\curl \bv\|_{\cc,\BH^m(\LP)}.
\end{equation}
Similarly, if $\bw \in \BH_0(\ddiv^0,\Omega)$ with $\curl(\cc\bw) \in \BH^m(\LP)$, then we have
\begin{equation}
\label{eq_shift_cc}
\|\bw\|_{\cc,\BH^{m+1}(\LP)}
\leq
\cD \frac{d_\Omega}{\vartheta_\Omega} \|\zz\curl (\cc\bw)\|_{\ee,\BH^m(\LP)},
\end{equation}
as can be seen by applying \cite[Theorem 2.2]{weber_1981a} to
$\cc \bw \in \BH(\ccurl,\Omega) \cap \BH_0(\ddiv^0,\mm,\Omega)$.
We also record the following result obtained by combining
\eqref{eq_shift_ee} and \eqref{eq_shift_cc}: if
$\bu \in \BH_0(\ccurl,\Omega) \cap \BH(\ddiv^0,\ee,\Omega)$ satisfies
$\curl(\cc\curl \bu) \in \BH^m(\LP)$, we have
\begin{equation}
\label{eq_double_shift}
\|\bu\|_{\ee,\BH^{m+2}(\LP)}
\leq
\left (\cD \frac{d_\Omega}{\vartheta_\Omega} \right )^2
\|\zz\curl (\cc\curl \bu)\|_{\ee,\BH^m(\LP)}
\end{equation}
for $0 \leq m \leq p-1$.

\begin{remark}[Smoothness assumption]
For the sake of simplicity, we assume that the coefficients are piecewise
analytic, which allow for regularity shifts for any $m \in \mathbb N$. In
turn, this allows us to establish our key results for any polynomial degree
$p \in \mathbb N$. On the other hand, for a fixed polynomial degree $p_\star \in \mathbb N$,
these smoothness assumptions can be weakened by simply requiring piecewise finite
regularity of the coefficients.
\end{remark}

\subsection{Curved tetrahedral mesh}

We consider a partition of $\Omega$ into a conforming mesh $\CT_h$ of
(curved) tetrahedral elements $K$ as in \cite[Assumption 3.1]{melenk_sauter_2020a}.
For $K \in \CT_h$ we denote by $\LF_K: \widehat K \to K$ the (analytic) mapping
between the reference tetrahedra $\widehat K$ and the element $K$. We further
assume that the mesh $\CT_h$ is conforming with the partition $\LP$, which
means that for each $K \in \CT_h$, there exists a unique $P \in \LP$ such that
$K \subset \overline{P}$. This last assumption means that the coefficients are
smooth inside each mesh cell.

\subsection{N\'ed\'elec finite element space}

In the remaining of this work, we fix a polynomial degree $p \geq 0$.
Then, following \cite[Chapter 15]{ern_guermond_2021a}, we introduce the N\'ed\'elec
polynomial space
\begin{equation*}
\BN_p(\widehat K) = \BP_p(\widehat K) + \bx \times \BP_p(\widehat K),
\end{equation*}
where $\BP_p(\widehat K) \eq \left (P_p(\widehat K)\right )^3$ and $P_p(\widehat K)$ stands for the
space of polynomials of degree less than or equal to $p$ defined over $\widehat K$. Classically,
the associate approximation space is obtained by mapping the N\'ed\'elec polynomial space to the
mesh cells through a Piola mapping, leading to
\begin{equation*}
\BW_h \eq
\left \{
\bv_h \in \BH_0(\ccurl,\Omega)
\quad \big | \quad
\left (\BD\LF_K^{-1}\right ) \left (\bv_h|_K \circ \LF_K^{-1} \right ) \in \BN_p(\widehat K)
\quad
\forall K \in \CT_h
\right \},
\end{equation*}
where $\BD\LF_K^{-1}$ is the Jacobian matrix of $\LF_K^{-1}$.

\subsection{High-order interpolation}

There exists an interpolation operator $\CJ_h: \BH^1(\LP) \cap \BH_0(\ccurl,\Omega) \to \BW_h$
and a constant $\cI$ solely depending on $p$, the regularity of the mesh and the
coefficients $\ee$ and $\mm$ such that
\begin{subequations}
\label{eq_interpolation}
\begin{equation}
\|\bv - \CJ_h \bv\|_{\ee,\Omega}
\leq
\cI \left (\frac{h}{d_\Omega}\right )^{p+1} \|\bv\|_{\ee,\BH^{p+1}(\LP)},
\end{equation}
whenever $\bv \in \BH^1(\LP)$ satisfies $\bv \in \BH^{p+1}(\LP)$ and
\begin{equation}
\|\curl(\bw - \CJ_h \bw)\|_{\cc,\Omega}
\leq
\cI \left (\frac{h}{d_\Omega}\right )^{p+1} \|\curl \bw\|_{\cc,\BH^{p+1}(\LP)}
\end{equation}
\end{subequations}
for all $\bw \in \BH_0(\ccurl,\Omega) \cap \BH^{1}(\LP)$ with $\curl \bw \in \BH^{p+1}(\LP)$.
The construction of such an interpolation operator is classical, and we refer the reader
to \cite[Chapters 13 and 17]{ern_guermond_2021a} and \cite[Chapter 8]{melenk_sauter_2020a}
for more details.

\begin{remark}[$p$-explicit interpolation estimates]
It is possible to obtain ``$p$-explicit'' versions of the estimates in \eqref{eq_interpolation},
with a constant $\cI$ independent of the polynomial degree. In our case, such estimates are not
useful because the dependency of $\cD$ on $p$ is unknown (or at least, not practically useful).
\end{remark}

\subsection{Sharp approximability estimates}

We are now ready to rigorously introduce the approximation factor $\gamma$.
Given $\bg \in \BL^2(\Omega)$, we denote by $\be^\star(\bg)$ the unique
element of $\BH_0(\ccurl,\Omega)$ such that
\begin{equation}
\label{eq_maxwell_weak}
-\omega^2(\ee\bv,\be^\star(\bg)) + (\cc\curl\bv,\curl \be^\star(\bg))
=
\omega(\ee\bv,\bg)
\end{equation}
for all $\bv \in \BH_0(\ccurl,\Omega)$. Note that the existence and uniqueness
of $\be^{\star}(\bg)$ follows from the assumption that $\delta > 0$, i.e. $\omega$ is not
a resonant frequency. Then, we introduce the approximation factor as
\begin{equation}
\label{eq_gamma}
\gamma \eq \sup_{\substack{
\bg \in \BH(\ddiv^0,\ee,\Omega)
\\
\|\bg\|_{\ee,\Omega} = 1
}}
\inf_{\bv_h \in \BW_h}
\enorm{\be^\star(\bg)-\bv_h}.
\end{equation}
The constant $\gamma$ plays a crucial role in showing the stability
of finite element discretizations, as detailed in \cite[\S2.2]{chaumontfrelet_nicaise_2019a}
for the Helmholtz equation. It appears in the context of a duality technique
often called the ``Schatz argument'', whereby $\bg$ is taken to be the finite element
error, and the function $\be^\star(\bg)$ is used to compensate for the negative $L^2$-term
of the bilinear form \cite{schatz_1974a}. A variation of the Schatz argument is also employed
in a posteriori error analysis \cite{chaumontfrelet_ern_vohralik_2021a,dorfler_sauter_2013a}.

Observing that we can choose $\bv_h = \bo$ in the infimum, a crude estimate
for the approximation factor is given by $\gamma \leq \cSE$ where
\begin{equation}
\label{eq_def_cSE}
\cSE \eq \sup_{\substack{
\bg \in \BH(\ddiv^0,\ee,\Omega)
\\
\|\bg\|_{\ee,\Omega} = 1
}}
\enorm{\be^\star(\bg)}.
\end{equation}
This upper bound is of little use in a priori error estimation where one needs
$\gamma$ to become small as $h \to 0$ in a duality argument
\cite{%
chaumontfrelet_2019a,%
chaumontfrelet_nicaise_2019a,%
chaumontfrelet_nicaise_pardo_2018a,%
ern_guermond_2018a,%
zhong_shu_wittum_xu_2009a}.
On the other hand, it is of interest in a posteriori error estimation, in particular, to obtain
guaranteed estimates \cite{chaumontfrelet_ern_vohralik_2021a}. Indeed, the constant
$\cSE$ is often easier to compute than sharper estimates since it only depends
on the frequency, the domain and the coefficients, and not on the mesh or the
discretization order.

%\section{Stability and basic regularity results}
\section{Stability}
\label{section_stability}

Here, we present a stability result, that follows from standard spectral theory.

\begin{theorem}[Stability]
\label{theorem_stability}
The estimates
\begin{equation}
\label{eq_stability_estimates}
\omega \|\be^\star(\bg)\|_{\ee,\Omega} \leq \frac{\omega}{\delta} \|\bg\|_{\ee,\Omega},
\qquad
\|\curl \be^\star(\bg)\|_{\cc,\Omega} \leq \frac{\omega}{\delta} \|\bg\|_{\ee,\Omega}
\end{equation}
hold true for all $\bg \in \BH(\ddiv^0,\ee,\Omega)$. In addition, we have
\begin{equation}
\label{eq_estimate_cSE}
\cSE \leq \frac{\omega}{\delta}.
\end{equation}
\end{theorem}

\begin{proof}
Let $\bg \in \BH(\ddiv^0,\ee,\Omega)$ and set $\be \eq \be^\star(\bg)$.
Since $\be,\bg \in \BH(\ddiv^0,\ee,\Omega)$, we may expand $\be$ and $\bg$
in the basis $\{\bphi_j\}_{j \geq 0}$ by letting $e_j \eq (\be,\bphi_j)$
and $g_j \eq (\be,\bphi_j)$. Then, picking $\bv = \bphi_j$ in \eqref{eq_maxwell_weak},
we see that
\begin{equation*}
|e_j|
=
\frac{\omega}{|\lambda_j-\omega^2|} |g_j|
\leq
\frac{1}{\delta}\frac{\omega}{\sqrt{\lambda_j}+\omega} |g_j|\quad
\text{ and }
\quad\left (\omega + \sqrt{\lambda_j}\right )|e_j|
\leq
\frac{\omega}{\delta} |g_j|.
\end{equation*}
Then, \eqref{eq_stability_estimates} follows from
\begin{equation*}
\enorm{\be}^2
=
\sum_{j \geq 0}
(\omega^2 + \lambda_j)|e_j|^2
\leq
\sum_{j \geq 0}
\left ( (\omega + \sqrt{\lambda_j})|e_j|\right )^2
\leq
\left (\frac{\omega}{\delta}\right )^2 \|\bg\|_{\ee,\Omega}^2,
\end{equation*}
and \eqref{eq_estimate_cSE} follows from \eqref{eq_stability_estimates}
recalling the definition of $\cSE$ in \eqref{eq_def_cSE}. \qed
\end{proof}

\section{Regularity splitting}
\label{section_splitting}

In this section, we provide a regularity splitting result
that is solely expressed in terms of $\cSE$ and $\cD$%
\footnote{The authors believe it is of interest to explicitly mention $c_s$
proofs, since at least in principle, the regularity splitting results may apply in cases
where $c_s$ is not obtain via Theorem \ref{theorem_stability}.}. We start with a basic regularity result,
obtained by combining Theorem \ref{theorem_stability} with the regularity
shift results from Section \ref{section_shift}.

\begin{lemma}[Basic regularity]
\label{lemma_basic_regularity}
For all $\bg \in \BH(\ddiv^0,\ee,\Omega)$, we have
\begin{equation}
\label{eq_basic_reg_be}
\omega \|\be^\star(\bg)\|_{\ee,\BH^1(\LP)}
\leq
\cSE \cD \frac{\omega d_\Omega}{\vartheta_\Omega} \|\bg\|_{\ee,\Omega},
\end{equation}
and
\begin{equation}
\label{eq_basic_reg_curl_be}
\|\curl \be^\star(\bg) \|_{\cc,\BH^1(\LP)}
\leq
(1+\cSE) \cD \frac{\omega d_\Omega}{\vartheta_\Omega}
\|\bg\|_{\ee,\Omega}.
\end{equation}
\end{lemma}

\begin{proof}
Pick $\bg \in \BH(\ddiv^0,\ee,\Omega)$ and set $\be \eq \be^\star(\bg)$.
We first observe that as $\be \in \BH_0(\ccurl,\Omega) \cap \BH(\ddiv^0,\ee,\Omega)$,
shift estimate \eqref{eq_shift_ee} implies that
\begin{equation*}
\omega \|\be\|_{\ee,\BH^1(\LP)}
\leq
\cD \frac{\omega d_\Omega}{\vartheta_\Omega}
\|\curl \be\|_{\cc,\Omega}
\end{equation*}
% and stability estimate \eqref{eq_stability_estimates} shows that
and the definition of the stability constant in \eqref{eq_def_cSE} shows that
\begin{equation*}
\|\be\|_{\ee,\BH^1(\LP)}
\leq
\cD \frac{d_\Omega}{\vartheta_{\Omega}} \|\curl \be\|_{\cc,\Omega}
\leq
\cD \cSE \frac{d_\Omega}{\vartheta_{\Omega}}\|\bg\|_{\ee,\Omega},
\end{equation*}
so that \eqref{eq_basic_reg_be} follows. On the other hand,
we establish \eqref{eq_basic_reg_curl_be} with \eqref{eq_shift_cc}, since
\begin{equation*}
\|\curl \be\|_{\cc,\BH^1(\LP)}
\leq
\cD \frac{d_\Omega}{\vartheta_\Omega} \|\zz\curl(\cc \curl \be)\|_{\ee,\Omega}
\leq
\cD \frac{d_\Omega}{\vartheta_\Omega} \left (\omega \|\bg\|_{\ee,\Omega} + \omega^2 \|\be\|_{\ee,\Omega}\right ),
\end{equation*}
using \eqref{eq_def_cSE} to estimate the last term. \qed
\end{proof}

The regularity results presented in Lemma \ref{lemma_basic_regularity}
suffice to obtain sharp estimates for the approximation factor
when $p=0$. For high-order elements, however, this is not sufficient.
As we only have a limited regularity assumption for the right-hand side $\bg$
in definition \eqref{eq_gamma} of $\gamma$, we may not expect more regularity
than established in Lemma \ref{lemma_basic_regularity} for the associated solution
$\be^\star(\bg)$. As shown in
\cite{chaumontfrelet_nicaise_2019a,lafontaine_spence_wunsch_2020a,melenk_sauter_2011a}
for the Helmholtz equation, the key idea is to introduce a ``regularity splitting'' of the
solution. Here, we shall adapt the approach of \cite{chaumontfrelet_nicaise_2019a}
to Maxwell's equations and consider the formal expansion
\begin{equation}
\label{eq_formal_expansion}
\be^\star(\bg)
=
\sum_{j \geq 0} \left (\frac{\omega d_\Omega}{\vartheta_\Omega} \right )^j \be_j^\star(\bg).
\end{equation}
After identifying the powers of $(\omega d_\Omega/\vartheta_\Omega)$
in \eqref{eq_maxwell_strong}, one sees that $\be_0^\star(\bg) \eq \bo$,
and that the other elements
$\be_j^\star(\bg) \in \BH_0(\ccurl,\Omega) \cap \BH(\ddiv^0,\ee,\Omega)$
are iteratively defined through
\begin{subequations}
\label{eq_def_be}
\begin{equation}
\curl \left (\cc \curl \be_1^\star(\bg) \right )
=
\frac{\vartheta_\Omega}{d_\Omega} \ee \bg,
\end{equation}
and
\begin{equation}
\curl \left (\cc \curl \be_j^\star(\bg)\right )
=
\left (\frac{\vartheta_\Omega}{d_\Omega}\right )^2\ee \be_{j-2}^\star(\bg)
\end{equation}
\end{subequations}
for $j \geq 2$. Note that the boundary value problems in \eqref{eq_def_be}
are well-posed, since $\|\curl \cdot\ \!\|_{\cc,\Omega}$
is a norm on $\BH_0(\ccurl,\Omega) \cap \BH(\ddiv^0,\ee,\Omega)$.
We first show that the iterates in the sequence exhibit increasing regularity.

\begin{lemma}[Increasing regularity of the expansion]
\label{lemma_increasing_regularity}
Let $\bg \in \BH(\ddiv^0,\ee,\Omega)$. For all $0 \leq j \leq p$, we have
$\be_j^\star(\bg) \in \BH^{j+1}(\LP)$ and $\curl \be_j^\star(\bg) \in \BH^j(\LP)$ with
\begin{equation}
\label{eq_bej}
\omega \|\be_j^\star(\bg)\|_{\ee,\BH^{j+1}(\LP)}
\leq
\cD^{j+1} \frac{\omega d_\Omega}{\vartheta_\Omega} \|\bg\|_{\ee,\Omega},
\end{equation}
and
\begin{equation}
\label{eq_curl_bej}
\|\curl \be_j^\star(\bg)\|_{\cc,\BH^j(\LP)}
\leq
\cD^j \|\bg\|_{\ee,\Omega}.
\end{equation}
\end{lemma}

\begin{proof}
Let $\bg \in \BH(\ddiv^0,\ee,\Omega)$. To ease the presentation, we set
$\be_j \eq \be_j^\star(\bg)$ for $j \geq 0$.
We start with \eqref{eq_curl_bej}. It obviously holds for $j = 0$
as $\be_0 \eq \bo$. For $j=1$, recalling \eqref{eq_def_be}, we have
\begin{equation*}
\|\curl \be_1\|_{\cc,\BH^1(\LP)}
\leq
\cD \frac{d_\Omega}{\vartheta_\Omega} \|\curl \left (\cc\curl \be_1\right )\|_{\zz,\Omega}
=
\cD \|\ee \bg\|_{\zz,\Omega}
=
\cD \|\bg\|_{\ee,\Omega}.
\end{equation*}
Then, assuming that \eqref{eq_curl_bej} holds up to some $j$,
\eqref{eq_shift_cc} and \eqref{eq_def_be} reveal that
\begin{multline*}
\|\curl \be_{j+2}\|_{\cc,\BH^{j+2}(\LP)}
\leq
\cD \frac{d_\Omega}{\vartheta_\Omega} \|\zz\curl (\cc \curl \be_{j+2})\|_{\ee,\BH^{j+1}(\LP)}
=
\cD \frac{\vartheta_\Omega}{d_\Omega} \|\be_j\|_{\ee,\BH^{j+1}(\LP)}
\\
\leq
\cD^2 \|\curl \be_j\|_{\cc,\BH^{j+1}(\LP)}
\leq
\cD^{j+2} \|\bg\|_{\ee,\Omega},
\end{multline*}
and \eqref{eq_curl_bej} follows by induction.

On the other hand, \eqref{eq_bej} is a direct consequence
of \eqref{eq_curl_bej}, since \eqref{eq_shift_ee} shows that
\begin{equation*}
\|\be_j\|_{\ee,\BH^{j+1}(\LP)}
\leq
\cD \frac{d_\Omega}{\vartheta_\Omega} \|\curl \be_j\|_{\cc,\BH^j(\LP)}
\leq
\frac{d_\Omega}{\vartheta_\Omega}
\cD^{j+1} \|\bg\|_{\ee,\Omega}.
\qquad \qed
\end{equation*}
\end{proof}

So far, expansion \eqref{eq_formal_expansion} is only formal, and we need to
truncate the expansion into a finite sum. To do so, we introduce,
for $\ell \geq 0$, the ``residual'' term
\begin{equation*}
\br_\ell^\star(\bg)
\eq
\be^\star(\bg)
-
\sum_{j=0}^\ell \left (\frac{\omega d_\Omega}{\vartheta_{\Omega}}\right )^j \be_j^\star(\bg)
\in
\BH_0(\ccurl,\Omega) \cap \BH(\ddiv^0,\ee,\Omega),
\end{equation*}
so that
\begin{equation}
\label{eq_finite_expansion}
\be^\star(\bg) = \sum_{j=0}^\ell
\left (
\frac{\omega d_\Omega}{\vartheta_\Omega}
\right )^j
\be_j^\star(\bg) + \br_\ell^\star(\bg).
\end{equation}
As we show next, these residuals have increasing regularity.

\begin{lemma}[Regularity of residual terms]
\label{lemma_residual_regularity}
For all $\bg \in \BH(\ddiv^0,\ee,\Omega)$ and $0 \leq \ell \leq p$, we have
$\br_\ell^\star(\bg) \in \BH^{\ell+1}(\LP)$ and $\curl \br_\ell^\star(\bg) \in \BH^{\ell+1}(\LP)$
with the estimates
\begin{equation}
\label{eq_br}
\omega \|\br_\ell^\star(\bg)\|_{\ee,\BH^{\ell+1}(\LP)}
\leq
\cSE \left (\cD \frac{\omega d_\Omega}{\vartheta_\Omega}\right )^{\ell+1}
\|\bg\|_{\ee,\Omega},
\end{equation}
and
\begin{equation}
\label{eq_curl_br}
\|\curl \br_\ell^\star(\bg)\|_{\cc,\BH^{\ell+1}(\LP)}
\leq
(1+\cSE) \left (\cD \frac{\omega d_\Omega}{\vartheta_\Omega} \right )^{\ell+1}
\|\bg\|_{\ee,\Omega}.
\end{equation}
\end{lemma}

\begin{proof}
For the sake of simplicity, we fix $\bg \in \BH(\ddiv^0,\ee,\Omega)$,
and set $\be \eq \be^\star(\bg)$ and $\br_\ell \eq \br_\ell^\star(\bg)$ for $\ell \geq 0$.
We have $\br_0 \eq \be$, so that \eqref{eq_br} and \eqref{eq_curl_br} hold for $\ell=0$
as a direct consequence of \eqref{eq_basic_reg_be} and \eqref{eq_basic_reg_curl_be}.

For the case $\ell=1$, simple computations show that
$\curl (\cc \curl \br_1) = \omega^2 \ee \be$.
Using \eqref{eq_shift_cc} and \eqref{eq_double_shift}, it then follows
that
\begin{align*}
\omega \|\br_1\|_{\ee,\BH^2(\LP)}
\leq
\cD^2 \frac{\omega d_\Omega^2}{\vartheta_\Omega^2} \|\zz\curl (\cc \curl \br_1)\|_{\ee,\Omega}
=
\cD^2 \left (\frac{\omega d_\Omega}{\vartheta_\Omega}\right )^2 \omega\|\be\|_{\ee,\Omega}
\end{align*}
and
\begin{equation*}
\|\curl \br_1\|_{\cc,\BH^2(\LP)}
\leq
\cD
\frac{d_\Omega}{\vartheta_\Omega} \|\zz\curl (\cc\curl \br_1)\|_{\ee,\BH^1(\LP)}
=
\cD \frac{\omega d_\Omega}{\vartheta_\Omega} \omega \|\be\|_{\ee,\BH^1(\LP)}
\end{equation*}
so that \eqref{eq_br} and \eqref{eq_curl_br} are also valid when $\ell=1$
recalling \eqref{eq_def_cSE} and \eqref{eq_basic_reg_be}.

For the general case, we first observe that
$\curl (\cc \curl \br_{\ell+2}) = \omega^2 \ee \br_\ell$.
Therefore, using \eqref{eq_shift_cc} and \eqref{eq_double_shift}, we have
\begin{equation*}
\omega\|\br_{\ell+2}\|_{\ee,\BH^{\ell+3}(\LP)}
\leq
\left (\cD\frac{\omega d_\Omega}{\vartheta_\Omega}\right )^2
\omega \|\br_\ell\|_{\ee,\BH^{\ell+1}(\LP)},
\end{equation*}
and
\begin{align*}
\|\curl \br_{\ell+2}\|_{\cc,\BH^{\ell+3}(\LP)}
&\leq
\cD \frac{d_\Omega}{\vartheta_\Omega}
\|\zz \curl (\cc \curl \br_{\ell+2})\|_{\ee,\BH^{\ell+2}(\LP)}
\\
&\leq
\cD \frac{d_\Omega}{\vartheta_\Omega} \omega^2 \|\br_{\ell}\|_{\ee,\BH^{\ell+2}(\LP)}
\leq
\left (\cD \frac{\omega d_\Omega}{\vartheta_\Omega}\right )^2
\|\curl \br_\ell\|_{\cc,\BH^{\ell+1}(\LP)},
\end{align*}
and the general case follows by induction. \qed
\end{proof}

\section{Sharp approximability estimates}
\label{section_approximability}

Equipped with the regularity splitting from Section \ref{section_splitting},
we are now ready to establish our main result, providing an upper bound for
the approximation factor $\gbaE$.

\begin{theorem}[Approximability estimate]
\label{theorem_approximability}
Assume that $\cD (\omega h/\vartheta_\Omega) \leq 1/2$. Then, the following estimate holds true
\begin{equation*}
\gbaE
\leq
\cI
\left (
2 \sqrt{2} \cD \frac{\omega h}{\vartheta_\Omega}
+
\sqrt{1+2\cSE^2}
\left (\cD\frac{\omega h}{\vartheta_\Omega}\right )^{p+1}
\right ).
\end{equation*}
\end{theorem}

\begin{proof}
We consider a right-hand side $\bg \in \BH(\ddiv^0,\ee,\Omega)$
and employ the notation $\be \eq \be^\star(\bg)$, $\be_j \eq \be_j^\star(\bg)$
for $j \geq 0$ and $\br_p \eq \br_p^\star(\bg)$.
% Let $\ell \geq 1$ and $j=1,\ldots,\ell$.
Recalling \eqref{eq_gamma}
and the finite expansion \eqref{eq_finite_expansion} for $\be$, it is sufficient
to provide upper bounds for the high-order interpolation error of $\be_j$ and $\br_\ell$.
For $ \be_j $, %the definition \eqref{eq_def_be} 
\eqref{eq_interpolation} and Lemma \ref{lemma_increasing_regularity} imply that
\begin{align*}
\omega \left (\frac{\omega d_\Omega}{\vartheta_\Omega}\right )^j
\|\be_j-\CJ_h \be_j\|_{\ee,\Omega}
&\leq
\cI
\omega \left (\frac{\omega d_\Omega}{\vartheta_\Omega}\right )^j
\left (\frac{h}{d_\Omega}\right )^{j+1} \|\be_j\|_{\ee,\BH^{j+1}(\LP)}
\\
&=
\cI
\frac{\vartheta_\Omega}{d_\Omega}
\left (\frac{\omega h}{\vartheta_\Omega}\right )^{j+1}
\|\be_j\|_{\ee,\BH^{j+1}(\LP)}
\leq
\cI
\left (\cD \frac{\omega h}{\vartheta_\Omega}\right )^{j+1} \|\bg\|_{\ee,\Omega},
\end{align*}
and
\begin{equation*}
\left (\frac{\omega d_\Omega}{\vartheta_\Omega}\right )^j
\|\curl (\be_j-\CJ_h \be_j)\|_{\cc,\Omega}
\leq
\cI
\left (\frac{\omega d_\Omega}{\vartheta_\Omega}\right )^j
\left (\frac{h}{d_\Omega}\right )^j \|\curl \be_j\|_{\cc,\BH^{j}(\LP)}
\leq
\cI
\left ( \cD \frac{\omega h}{\vartheta_\Omega} \right )^j \|\bg\|_{\ee,\Omega},
\end{equation*}
and since $\cD (\omega h/\vartheta_\Omega) \leq 1$, we get
\begin{equation*}
\left (\frac{\omega d_\Omega}{\vartheta_\Omega}\right )^j
\enorm{\be_j-\CJ_h \be_j}
\leq
\cI
\sqrt{2} \left ( \cD \frac{\omega h}{\vartheta_\Omega} \right )^j \|\bg\|_{\ee,\Omega}.
\end{equation*}

Similarly, using to Lemma \ref{lemma_residual_regularity},
we have for the residual $\br_p$
\begin{align*}
\omega \|\br_p-\CJ_h \br_p\|_{\ee,\Omega}
\leq
\cI \left ( \frac{h}{d_\Omega} \right )^{p+1} \omega\|\br_k\|_{\ee,\BH^{p+1}(\LP)}
\leq
\cI \cSE \left (\cD \frac{\omega h}{\vartheta_\Omega} \right )^{p+1} \|\bg\|_{\ee,\Omega}
\end{align*}
and
\begin{align*}
\|\curl (\br_p-\CJ_h \br_p)\|_{\cc,\Omega}
\leq
\cI
\left (\frac{h}{d_\Omega}\right )^{p+1} \|\curl \br_k\|_{\cc,\BH^{p+1}(\LP)}
\leq
\cI
(1+\cSE)
\left ( \cD \frac{\omega h}{\vartheta_\Omega} \right )^{p+1}
\|\bg\|_{\ee,\Omega},
\end{align*}
and hence
\begin{equation*}
\enorm{\br_p-\CJ_h \br_p}
\leq
\cI \sqrt{1+2\cSE^2}
\left ( \cD \frac{\omega h}{\vartheta_\Omega} \right )^{p+1}
\|\bg\|_{\ee,\Omega}.
\end{equation*}

Then, recalling the expansion \eqref{eq_finite_expansion}, the above estimates show that
\begin{equation*}
\enorm{\be - \CJ_h \be}
\leq
\cI \left (
\sqrt{2}
\sum_{j=1}^p \left (\cD \frac{\omega h}{\vartheta_\Omega}\right )^j
+
\sqrt{1+2\cSE^2} \left ( \cD \frac{\omega h}{\vartheta_\Omega} \right )^{p+1}
\right )
\|\bg\|_{\ee,\Omega}.
\end{equation*}
Finally, the result follows by
\begin{align*}
\sum_{j=1}^p \left (\cD \frac{\omega h}{\vartheta_\Omega}\right )^j
=
\left (\cD \frac{\omega h}{\vartheta_\Omega}\right )\sum_{j=0}^{p-1} \left (\cD \frac{\omega h}{\vartheta_\Omega}\right )^j
&=
\left (\cD \frac{\omega h}{\vartheta_\Omega}\right )
\frac{1 - \left (\cD \frac{\omega h}{\vartheta_\Omega}\right )^p}
{1 - \left (\cD \frac{\omega h}{\vartheta_\Omega}\right )}
\leq
2\cD \frac{\omega h}{\vartheta_\Omega}.
\end{align*}
\qed
\end{proof}

We conclude our work with a simplified version of Theorem \ref{theorem_approximability}
that is easier to read, but not as explicit in how the estimate depends on $\cSE$, $\cD$
and $\cI$. We skip the proof as it immediately follows from Theorems \ref{theorem_stability}
and \ref{theorem_approximability}.

\begin{corollary}[Simplified approximability estimate]
\label{corollary_approximability}
There exist positive constants $c$ and $C$ solely depending on $\cSE$, $\cD$ and $\cI$
such that whenever $\omega h/\vartheta_\Omega \leq c$, we have
\begin{equation*}
\gbaE
\leq
C \left (
\frac{\omega h}{\vartheta_\Omega}
+
\frac{\omega}{\delta}\left (\frac{\omega h}{\vartheta_\Omega}\right )^{p+1}
\right ).
\end{equation*}
\end{corollary}

\bibliographystyle{amsplain}
\bibliography{bibliography}

\end{document}